\documentclass[11pt]{amsart}

\usepackage{amssymb}
\usepackage{mathrsfs}
\usepackage{amsfonts}
\usepackage{latexsym,amsmath,amsthm,amssymb,amsfonts}
\usepackage[usenames]{color}
\usepackage{xspace,colortbl}
\usepackage{graphicx}
\usepackage{tipa}

\newcommand{\be}{\begin{equation}}
\newcommand{\ee}{\end{equation}}
\newcommand{\beq}{\begin{eqnarray}}
\newcommand{\eeq}{\end{eqnarray}}

\def\begeq{\begin{equation}}
\def\endeq{\end{equation}}

\def\odot{\setbox0=\hbox{$\bigcirc$}\relax \mathbin {\hbox
to0pt{\raise.5pt\hbox to\wd0{\hfil $\wedge$\hfil}\hss}\box0 }}

\numberwithin{equation} {section}

\numberwithin{equation}{section}
\newtheorem{theorem}{\bf Theorem}[section]

\newtheorem{definition}[theorem]{\bf Definition}
\newtheorem{lemma}[theorem]{\bf Lemma}
\newtheorem{example}[theorem]{\bf Example}
\newtheorem{corollary}[theorem]{\bf Corollary}
\newtheorem{conjecture}[theorem]{\bf Conjecture}

\allowdisplaybreaks

\begin{document}

\title[Rigidity theorems
for the entire solutions of 2-Hessian equation]
{Rigidity theorems for the entire solutions of 2-Hessian equation}

\author{Li Chen and Ni Xiang$^\ast$}
\address{
Faculty of Mathematics and Statistics, Hubei University, Wuhan
430062, China. } \email{chernli@163.com, nixiang@hubu.edu.cn}

\thanks{This research was also supported in part by funds from Hubei Provincial Department of Education and Hubei Key
Laboratory of Applied Mathematics (Hubei University).}

\thanks{$\ast$ the corresponding author}

\date{}
%\maketitle
\begin{abstract}
In this paper, we prove some rigidity theorems
for the entire 2-convex solutions of 2-Hessian equation in Euclidean space.
As an application, we obtain a Bernstein type theorem for global special Lagrangian graphs.
\end{abstract}

\maketitle {\it \small{{\bf Keywords}: Rigidity theorem,
2-Hessian equation, 2-convex.}

{{\bf MSC}: Primary 35J60, Secondary
35B45.}
}

\section{Introduction}

It is very interesting to consider the Liouville theory for the entire solutions
$u$ in $n$-dimensional Euclidean spaces
of the following equations,
\begin{eqnarray}\label{sg_k}
\sigma_ {k} (D^2 u(x)) = 1.
\end{eqnarray}
Here $\sigma_{k}(D^2 u(x))$ is the $k$-Hessian operator of $u$ and is defined as follows. Let $\sigma_k(\lambda)$ be the
$k$-th elementary symmetric function of $\lambda \in \mathbb{R}^n$.
Then $\sigma_ {k} (D^2 u(x))=\sigma_ {k} (\lambda[D^2 u])$, where
$\lambda[D^2 u]$ are the eigenvalues of the
Hessian matrix, $D^2 u$, of a function $u$ defined in $\mathbb{R}^n$.
Alternatively, it can be written as the sum of the $k \times k$
principal minors of $D^ 2 u$.

To ensure the ellipticity of \eqref{sg_k}, we have to restrict the class of functions
and domains.
\begin{definition}
A function $u \in C^2(\mathbb{R}^n)$ is called $k$-convex
if $\lambda [D^2 u]=(\lambda_1[D^2 u], ..., \lambda_n[D^2 u])$
 belongs to $\Gamma_k$ for all $x \in \mathbb{R}^n$, where $\Gamma_k$ is the Garding's cone
\begin{eqnarray*}\label{cone}
\Gamma_{k}=\{\lambda \in \mathbb{R} ^n: \sigma_{j}(\lambda)>0, \forall 1\leq j \leq k\}.
\end{eqnarray*}
\end{definition}

We review some rigidity theorems
for the entire solutions of the above equations. For $k = 1$, \eqref{sg_k} is a linear
equation.  Its entire convex solution must be a quadratic polynomial.
For $k = n$,  the Monge-Amp\`ere equation, a well-known theorem due to
J$\ddot{o}$rgens \cite{Jo} ($n=2$),  Calabi \cite{Ca} ($n = 3, 4, 5$) and Pogorelov \cite{Po1} \cite{Po2}
($n\geq 2$) asserts that that any entire strictly convex solution must be a
quadratic polynomial. A simpler and more analytical proof, along the lines of affine geometry, of the
theorem was later given by Cheng and Yau \cite{CYau}. It was proven by Trudinger and
Wang \cite{TrW} that the only open convex subset $\Omega$ of $\mathbb{R}^n$ which admits a convex $C^2$
solution of $\det(D^2u) = 1$ in $\Omega$ with $\lim_{x\rightarrow\partial \Omega} u(x) = \infty$ is $\Omega= \mathbb{R}^n$.
In 2003, Caffarelli and Li, \cite{CL} extended the theorem of J$\ddot{o}$rgens, Calabi and Pogorelov based on the theory of Monge-Amp\`ere equations \cite{Ca1,Ca2}.

For $k=2$, Chang and Yuan \cite{CY} have proved that, if the lower bound
\begin{eqnarray*}
D^2 u(x)\geq \bigg[\delta-\sqrt{\frac{2}{n(n-1)}}\bigg] I
\end{eqnarray*}
for any $\delta>0$, then the entire convex solutions of the equation \eqref{sg_k} must be quadratic
polynomials. And they also guess their result should still be true under the semiconvexity assumption
$D^2 u\geq-K I$ with arbitrarily large $K$, even for general equation \eqref{SQ}
with $2\leq k \leq n-1$. Their conjecture holds true in the case when $n=3$ and $k=2$ (see Theorem
1.3 in \cite{Yu}). Here a different transformation and the geometric measure theory were
employed.

It would be interesting to see if the rigidity theorem remains valid under some assumptions for general $k$.  Bao, Chen, Guan and Ji \cite{BC} obtained that, strictly convex entire solutions of \eqref{sg_k},
satisfying a quadratic
growth are quadratic polynomials. Here the quadratic growth is defined as follows,
\begin{definition}
A function $u: \mathbb{R}^n\rightarrow \mathbb{R}$ satisfies the quadratic growth if there are
some positive constants $b, c$ and sufficiently large $R$, such that,
\begin{eqnarray}\label{gro}
u(x)\geq b|x|^2-c, \quad \mbox{for} \quad |x|\geq R.
\end{eqnarray}
\end{definition}
In \cite{BC}, the authors raised a question if the rigidity theorem remains valid 
under weaker or without growth assumptions, or for $k$-convex solutions. 
Recently, Li, Ren and Wang \cite{LRW} have obtained that strictly convex assumption
in \cite{BC} can be reduced to $(k+1)$-convexity.

There are also some Liouville type theorems for complex Hessian equations. Dinew and Kolodziej \cite{DK}
have proved a Liouville type theorem for entire maximal $m$-subharmonic functions in $\mathbb{C}^n$
with bounded gradient recently. Li and Sheng \cite{LS} considered
complex Monge-Amp\`ere equations $\det(u_{i\overline{j}})=1$ in $\mathbb{C}^n$ and obtained
the Liouville theorem under the assumption of the quadric
growth
\begin{equation*}
C^{-1} ( 1 + |z|^2) \leq u \leq C (1 + |z| ^2),\ \ \mbox{as}\ \  |z| \rightarrow  \infty,
\end{equation*}
for some $C>0$.

Thus, it is interesting to ask whether we can relax $(k+1)$-convexity in \cite{LRW} further, even we reduce it to 
$k$-convexity? Fortunately, we can make some progresses on this problem
for the entire solutions of the 2-Hessian equation
\begin{equation}\label{SQ}
\sigma_2(D^2u(x))=1.
\end{equation}

Our main theorems are stated as follows.
\begin{theorem}\label{main}
Given any nonnegative constant $A$, any entire $2$-convex solution $u \in C^4(\mathbb{R}^n)$ of the
equation \eqref{SQ} define in $\mathbb{R}^n$
satisfying $\sigma_{3}(D^2 u(x))\geq -A$ and a quadratic growth \eqref{gro} is a quadratic polynomial.
\end{theorem}

For $n=3$, the assumption $\sigma_{3}(D^2 u(x))\geq -A$ is redundant.
\begin{theorem}\label{main1}
Any entire $2$-convex solution $u \in C^4(\mathbb{R}^3)$ of the equation \eqref{SQ} define in $\mathbb{R}^3$
satisfying a quadratic growth \eqref{gro} is a quadratic polynomial.
\end{theorem}

The following example which is given by Warren in \cite{Wa} shows that
the quadratic growth in Theorem \ref{main1} is necessary.
\begin{example}
When $n=3$ the non-polynomial function
$$u(x, y, t)=(x^2+y^2)e^{t}+\frac{1}{4}e^{-t}-e^{t}$$
is a $2$-convex solution of the equation \eqref{SQ}, but this function does not have a quadratic growth.
\end{example}

Therefore, it is very interesting to ask if the assumption $\sigma_{3}(D^2 u(x))\geq -A$ is redundant for any $n$. So we propose the following conjecture:
\begin{conjecture}
Any entire $2$-convex solution $u \in C^4(\mathbb{R}^n)$ of the
equation \eqref{SQ} define in $\mathbb{R}^n$
with a quadratic growth \eqref{gro} is a quadratic polynomial.
\end{conjecture}

Equations (\ref{sg_k}) naturally appear in many interesting geometric problems such as Minkowski problem which is connected with Monge-Amp\`ere equation and the $k^{th}$-Weingarten curvature problem.
For $k = 1, 2 $ and $ n$, $\sigma _k$ corresponds to mean, scalar, and Gauss curvatures respectively.
When $n=3$ and $k=2$, the equation \eqref{sg_k} arises in special Lagrangian
geometry \cite{HL}. The special Lagrangian equation is
\begin{eqnarray}\label{Lag}
\sum_{i=1}^{n}\mbox{arctan} \lambda_i=\theta,
\end{eqnarray}
here $\lambda_i$ are eigenvalues of $D^2 u$ for
\begin{eqnarray*}
\theta \in (-\frac{n}{2}\pi, \frac{n}{2}\pi)
\end{eqnarray*}
a constant. Equation \eqref{Lag} originates from special Lagrangian geometry \cite{HL}. The
(Lagrangian) graph $(x, Du(x)) \subset \mathbb{R}^n\times \mathbb{R}^n$ is called special when the argument
of the complex number $(1+\sqrt{-1}\lambda_1)\cdot\cdot\cdot(1+\sqrt{-1}\lambda_n)$
is constant $\theta$ or $u$ satisfies \eqref{Lag}, and it is special if and only if $(x, Du(x))$ is a minimal
surface in $\mathbb{R}^n\times \mathbb{R}^n$, see Theorem 2.3 and Proposition 2.17 in \cite{HL}.

Before we state our corollary, we need to mention some related work on Bernstein type results
for global special Lagrangian graphs.
Fu \cite{Fu} showed that when $n=2$ and $\theta\neq 0$ all solutions of \eqref{Lag} are quadratic.
Jost-Xin \cite{JX} treated the problem using
a different method under the assumption that $|D^2u|$ is bounded.
When $n=2$ and $\theta=0$ the equation \eqref{Lag} becomes simply the Laplace equation,
which admits well-known non-polynomial solutions. Yuan \cite{Yu} showed that all
convex solutions to special Lagrangian equations are quadratic.

The critical phase for special Lagrangian equations is
$$|\theta|=\frac{n-2}{2}\pi.$$
Yuan \cite{Yu1} has shown that for $|\theta|>\frac{n-2}{2}\pi$, all entire solutions are
quadratic. In \cite{HL} when $n=3$, the critical equation
\begin{eqnarray*}
\sum_{i=1}^{n}\mbox{arctan} \lambda_i=\frac{\pi}{2}
\end{eqnarray*}
is equivalent to the equation \eqref{SQ}. Thus, we can get the following Bernstein type theorem from Theorem \ref{main1}.

\begin{corollary}
Suppose $M=(x, Du(x))$ is a minimal surface in $\mathbb{R}^3\times \mathbb{R}^3$ and
$u \in C^{4}(\mathbb{R}^3)$ is $2$-convex and satisfies a quadratic growth \eqref{gro},
then $M$ is a plane.
\end{corollary}

The main technique employed in this paper was motivated by the recent progresses on the interior estimates of
2-Hessian equation made by Guan and Qiu \cite{Guan1} \cite{Qiu}. In \cite{Guan1}, Guan and Qiu obtained
the interior estimates of
2-Hessian equation by exploiting some special properties of the eigenvalues of $D^2 u$
under the assumptions that
$u$ is 2-convex and $\sigma_3(D^2 u)\geq -A$. Using these special properties together
with Maximum Principle, we can obtain
Theorem \ref{main}. For $n=3$, Qiu \cite{Qiu} showed $\sigma_3(D^2 u)\geq -A$ is not
needed and the technique in \cite{Qiu} can be used to prove Theorem \ref{main1}.

\textbf{Acknowledgement:} The authors would like to thank Dr.Chuanqiang Chen for pointing out
some mistakes in the paper, bringing our attention to the very important examples in \cite{Wa}
and giving some suggestive comments, and the first author expresses
his deep gratitude to Prof. Guofang Wang for some important suggestions on this paper, and
he also thanks Prof. Yu Yuan for some comments.

\section{The proof of Theorem \ref{main}}

Let $W=(W_{ij})$ be a symmetric tensor and
$\sigma_{k}(W)=\sigma_k(\lambda[W])$, where $\lambda[W]$ denotes the eigenvalues of the
$W$. Similarly, we say $W \in \Gamma_2$ if $\lambda[W] \in \Gamma_2$, which also
means $\sigma_1(W)>0$, $\sigma_2(W)>0$. It follows from \cite{CNS}, if $W \in \Gamma_2$,
then $\sigma_{2}^{ij}=\frac{\partial \sigma_2}{\partial W_{ij}}(W)$ is positive definite.
We first recall the following important Lemma in \cite{CQ}.
\begin{lemma}\label{le1}
Suppose $W \in \Gamma_2$ is diagonal and $W_{11}\geq\cdot\cdot\cdot\geq W_{nn}$, if $\xi_{ij}$ is symmetric and
$$\sum_{i=2}^{n}\sigma_{2}^{ii}\xi_{ii}+\sigma_{2}^{11}\xi_{11}=\eta,$$ then
$$-\sum_{i\neq j}\xi_{ii}\xi_{jj}\geq \frac{n-1}{2\sigma_{2}(W)}
\frac{[2\sigma_{2}(W)\xi_{11}-W_{11}\eta]^2}{[(n-1)W_{11}^2+2(n-2)\sigma_{2}(W)]}
-\frac{\eta^2}{2\sigma_{2}(W)}.$$
\end{lemma}

For our case when $u \in C^4(\mathbb{R}^n)$, let $W=D^2 u$, $\sigma_2(D^2u)=1$, $\xi_{ij}=u_{ij1}$. Thus, $\eta=0$ and
 we obtain the following corollary directly.
\begin{corollary}\label{co1}
Let $u \in C^4(\mathbb{R}^n)$ be a 2-convex solution of \eqref{SQ}, then
$$-\sum_{i\neq j}u_{ii1}u_{jj1}\geq
\frac{2(n-1)u_{111}^2}{[(n-1)u_{11}^2+2(n-2)]}.$$
\end{corollary}

Next, we recall the following Lemma 3 in \cite{Guan1}. For completeness, we give the proof here.
\begin{lemma}\label{le2}
Under the same assumption as in Lemma \ref{le1}, and in addition that there exists
a positive constant
\begin{eqnarray}\label{le2-0}
a\leq \sqrt{\frac{\sigma_{2}(W)}{3(n-1)(n-2)}}
\end{eqnarray}
(if $n=2$, $a>0$ could be arbitrary), such that
\begin{eqnarray}\label{le2-1}
\sigma_{3}(W+aI)\geq 0,
\end{eqnarray}
then
\begin{eqnarray}\label{le2-2}
\frac{7}{6}\sigma_{2}(W)\geq (\sigma_{2}(W)+\frac{(n-1)(n-2)}{2}a^2)\geq \frac{5}{6}\sigma_{2}^{11}(W)W_{11},
\end{eqnarray}
provided that $W_{11}>6(n-2)a$. Furthermore, for any $j \in \{2,...,n\}$,
\begin{eqnarray}\label{le2-3}
|W_{jj}|\leq (n-1)^2a+\frac{7(n-1)\sigma_{2}(W)}{5W_{11}}.
\end{eqnarray}
\end{lemma}

\begin{proof}
Since $W+aI \in \overline{\Gamma}_3$, it follows that
$$\widetilde{W}=(W_{22}+a, ..., W_{nn}+a)\in \overline{\Gamma}_2.$$
Thus,
$$\sigma_{2}(\widetilde{W})\geq 0,$$
which means
$$\sigma_{2}(W)-\sigma_{2}^{11}(W)W_{11}+(n-2)a\sigma_{2}^{11}(W)+\frac{(n-1)(n-2)}{2}a^2\geq 0.$$
Noting that $$\frac{(n-1)(n-2)}{2}a^2\leq \frac{\sigma_{2}(W)}{6}$$
and $$W_{11}-(n-2)a\geq \frac{5}{6}W_{11},$$ we have
\begin{eqnarray*}
\frac{7}{6}\sigma_{2}(W)&\geq&\frac{(n-1)(n-2)}{2}a^2+\sigma_{2}(W)
\\&\geq&(W_{11}-(n-2)a)\sigma_{2}^{11}(W)\\&\geq& \frac{5}{6}W_{11}\sigma_{2}^{11}(W).
\end{eqnarray*}
Since $W$ is diagonal and $W+aI \in \overline{\Gamma}_3$, thus
$$\sum_{j=3}^{n}(W_{jj}+a)\geq 0.$$
So we have
$$\sigma_{2}^{11}(W+aI)=W_{22}+a+\sum_{j=3}^{n}(W_{jj}+a)\geq W_{22}+a.$$
and
$$W_{22}\leq (n-2)a+\sigma_{2}^{11}(W).$$
On the other hand, as $W \in \Gamma_2$ and $W_{11}\geq...\geq W_{nn}$,
$$0\leq\sigma_{2}^{11}(W)\leq (n-2)W_{22}+W_{nn}$$
and
$$-W_{jj}\leq-W_{nn}\leq (n-2)W_{22}$$
for $2\leq j\leq n$. Thus, we obtain
$$|W_{jj}|\leq(n-2)W_{22}.$$
Therefore,
$$|W_{jj}|\leq(n-2)W_{22}\leq (n-1)^2a+(n-1)\sigma_{2}^{11}(W)\leq (n-1)^2a
+\frac{7(n-1)}{5}\frac{\sigma_{2}(W)}{W_{11}}.$$
\end{proof}

Using the above lemma for the solution of the equation \eqref{SQ}, we can have
\begin{corollary}
Let $u$ be a 2-convex solution of \eqref{SQ}. Assume $D^2u$ is diagonal and
$u_{11}\geq\cdot\cdot\cdot\geq u_{nn}$, there exists a constant $A$ sufficiently large  such that
$$\sigma_3(D^2u)\geq -A,$$
and
\begin{eqnarray}\label{u11}
u_{11}\geq \frac{6(n-2)}{n}A,
\end{eqnarray}
then
\begin{eqnarray}\label{co2-1}
\sigma_2(D^2u)\geq \frac{5}{7}\sigma_{2}^{11}(D^2 u)u_{11},
\end{eqnarray}
and
\begin{eqnarray}\label{co2-2}
|u_{jj}|\leq (n-1)^2 \sqrt{\frac{2A}{n(n-1)\sigma_1}}+\frac{7(n-1)}{5u_{11}}
\leq (n-1)^2 \sqrt{\frac{2A}{n(n-1)u_{11}}}+\frac{7(n-1)}{5u_{11}}.
\end{eqnarray}
\end{corollary}

\begin{proof}
We may pick $a=\sqrt{\frac{2A}{n(n-1)\sigma_{1}(D^2 u)}}$. Since $u$ is 2-convex,
$u_{11}\leq\sigma_1$. Then clearly,
\begin{eqnarray*}
a=\sqrt{\frac{2A}{n(n-1)\sigma_{1}}}\leq \sqrt{\frac{2A}{n(n-1)u_{11}}}
\leq\sqrt{\frac{1}{3(n-1)(n-2)}}\leq\sqrt{\frac{\sigma_2}{3(n-1)(n-2)}}
\end{eqnarray*}
in view of \eqref{u11}. Meanwhile,
\begin{eqnarray*}
\sigma_{3}(W+aI)&=&\sigma_{3}(W)+na\sigma_{2}(W)
+\frac{n(n-1)}{2}a^2\sigma_{1}(W)+\frac{n(n-1)(n-2)}{6}a^3
\\&\geq& \sigma_{3}(W)+\frac{n(n-1)}{2}a^2\sigma_{1}(W)\\&\geq&0,
\end{eqnarray*}
which guarantees the condition \eqref{le2-1} is satisfied. Lastly, if we choose $A$ sufficiently large, we have from \eqref{u11}
$$u_{11}^{\frac{3}{2}}\geq 6(n-2)\sqrt{\frac{2A}{n(n-1)}},$$
which implies
$$u_{11}\geq 6(n-2)\sqrt{\frac{2A}{n(n-1)u_{11}}}\geq 6(n-2)\sqrt{\frac{2A}{n(n-1)\sigma_{1}}}=6(n-2)a.$$
Then, this corollary
follows from Lemma \ref{le2} directly.
\end{proof}

To prove Theorem \ref{main}, we need the following key Lemma.
\begin{lemma}\label{le3}
Let $\Omega$ be a bounded open set in $\mathbb{R}^n$,
we consider the Drichlet problem of the 2-Hessian equation
\begin{equation}\label{SQ-}
\left\{
\begin{aligned}
&\sigma_2(D^2u(x))=1, \quad x \in \Omega; \\
&u=0, \quad x \in \partial\Omega.
\end{aligned}
\right.
\end{equation}
Assume $$\sigma_{3}(D^2 u)\geq -A$$ for some positive constant $A$.
Then, for any $2$-convex solution $u \in C^4(\Omega)\cap C(\overline{\Omega})$, we have the Pogorelov type estimate,
\begin{eqnarray}\label{est}
\max_{x \in \Omega}(-u(x))^{\alpha}|D^2 u(x)|\leq C
\end{eqnarray}
for sufficiently large $\alpha>0$. Here $\alpha$ and
 $C$ only depend on $A$, $n$, the diameter of the domain $\Omega$.
\end{lemma}

\begin{proof}
First, we translate the coordinate system such that $\Omega$ contains the coordinate origin.
Using the Comparison principle (see Theorem 17.1 in Page 443 of \cite{GT}), there exists the function
$$w=\frac{1}{\sqrt{2n(n-1)}}|x|^2-a$$
such that
$$w\leq u\leq 0,$$
where $a$ depends on the diameter of the domain. Since $u=0$ on $\partial \Omega$, we have $u\leq 0$ in $\Omega$.
Thus, $|u|$ can be controlled by the diameter of the domain $\Omega$.

Since $\sigma_1 (D^2u)>0$, we obtain $$|D^2 u(x)|\leq (n-1)\max_{\xi\in\mathbb{S}^{n-1}}u_{\xi\xi}(x).$$
So we need only to estimate
\begin{eqnarray*}
\max_{(x, \xi)\in \Omega\times \mathbb{S}^{n-1}}(-u(x))^{\alpha}u_{\xi\xi}\leq C.
\end{eqnarray*}
Now we consider the function for $x \in \Omega$, $\xi \in \mathbb{S}^{n-1}$
\begin{eqnarray*}
\widetilde{P}(x, \xi)=\alpha\log (-u)+\log \max \{u_{\xi\xi}(x), 1\}+\frac{1}{2}|x|^2,
\end{eqnarray*}
where $\alpha$ is a constant to be determined later.
Since $u=0$ on $\partial \Omega$, the maximum of $\widetilde{P}$ is attained
in some interior point $x_0$ of $\Omega$ and some $\xi(x_0) \in \mathbb{S}^{n-1}$.
Choose smooth orthonormal local frames $e_1, \ldots, e_n$
about $x_0$ such that $\xi(x_0)=e_1$ and $\{u_{ij} (x_0)\}$ is diagonal. Set
$$u_{11}(x_0)\geq u_{22}(x_0)\geq...\geq u_{nn}(x_0).$$
We may also assume that $u_{11}(x_0)\geq 1$ is sufficiently large.
Then we consider the function
\begin{eqnarray*}
P(x)=\alpha\log (-u)+\log u_{11}+\frac{1}{2}|x|^2.
\end{eqnarray*}
Note that $x_0$ is also a maximum point of $P$. We want to estimate $P(x_0)$.

At the maximum point $x_0$,
\begin{eqnarray}\label{1-diff}
0=P_i=\frac{\alpha u_i}{u}+\frac{u_{11i}}{u_{11}}+ x_i.
\end{eqnarray}
Noticing
\begin{eqnarray*}
P_{ij}=\frac{\alpha u_{ij}}{u}-\frac{\alpha u_i u_j}{u^2}
+\frac{u_{11ij}}{u_{11}}-\frac{u_{11i}u_{11j}}{u_{11}^2}+\delta_{ij},
\end{eqnarray*}
we can get at $x_0$
\begin{eqnarray*}
0\geq \sigma_{2}^{ij}P_{ij}&=&\frac{2\alpha}{u}
-\frac{\alpha\sigma_{2}^{ii}u_{i}^2}{u^2}
+\frac{\sigma_{2}^{ij}u_{11ij}}{u_{11}}-\frac{\sigma_{2}^{ii}u_{11i}^2}{u_{11}^2}
\\&&+(n-1)\sigma_1,
\end{eqnarray*}
in view of $\sigma_2=1$. Differential equation \eqref{SQ} in $k$-th variable,
$$\sigma_{2}^{ij}u_{ijk}=0.$$
Taking twice derivative of the equation \eqref{SQ},
\begin{eqnarray*}
\sigma_{2}^{ij, kl}u_{kl1}u_{ij1}+\sigma_{2}^{ij}u_{ij11}=0,
\end{eqnarray*}
which means
\begin{eqnarray}\label{le3-1}
\sigma_{2}^{ij}u_{ij11}=\sum_{i\neq j}u^{2}_{ij1}-\sum_{i\neq j}u_{ii1}u_{jj1}.
\end{eqnarray}
Now we want to estimate the second term on the right side of the above equality.
Assume that $u_{11}\geq \sqrt{\frac{6(n-2)}{n-1}}$ at $x_0$,
otherwise our Lemma holds true. Then, using Corollary \ref{co1}, we obtain
\begin{eqnarray*}
-\sum_{i\neq j}u_{ii1}u_{jj1}\geq \frac{3}{2}\frac{u_{111}^2}{u_{11}^2},
\end{eqnarray*}
which together with the inequality \eqref{co2-1}
\begin{eqnarray*}
1=\sigma_2\geq \frac{5}{7}\sigma^{11}_{2}u_{11},
\end{eqnarray*}
implies
\begin{eqnarray}\label{102901}
-\sum_{i\neq j}u_{ii1}u_{jj1}\geq \frac{\frac{15}{14}\sigma^{11}_{2}u_{111}^2}{u_{11}}.
\end{eqnarray}
Inserting \eqref{102901} into \eqref{le3-1} we have
\begin{eqnarray*}
\frac{\sigma_{2}^{ij}u_{11ij}}{u_{11}}
-\frac{\sigma_{2}^{ii}u^{2}_{11i}}{u^{2}_{11}}&\geq&\frac{2\sum_{i\neq 1}u^{2}_{11i}}{u_{11}}
-\frac{\sigma_{2}^{ii}u^{2}_{11i}}{u^{2}_{11}}+\frac{\frac{15}{14}\sigma^{11}_{2}u_{111}^2}{u_{11}^2}\\&=&
\frac{\sum_{i\neq 1}(2u_{11}-\sigma_{2}^{ii})u^{2}_{11i}}{u_{11}^2}
+\frac{\frac{1}{14}\sigma^{11}_{2}u_{111}^2}{u_{11}^2}.
\end{eqnarray*}
Then,
\begin{eqnarray*}
\sigma_{2}^{ij}P_{ij}&\geq&\frac{2\alpha }{u}
-\frac{\alpha\sigma_{2}^{ii}u_{i}^2}{u^2}+\frac{\sum_{i\neq 1}(2u_{11}-\sigma_{2}^{ii})u^{2}_{11i}}{u_{11}^2}+
\frac{\frac{1}{14}\sigma^{11}_{2}u_{111}^2}{u_{11}^2}
\\&&+(n-1)\sigma_1.
\end{eqnarray*}
In view of \eqref{1-diff} and the Cauchy-Schwarz inequality, we have
\begin{eqnarray*}
-(\frac{u_i}{u})^2\geq -\frac{2}{\alpha^2}\frac{u_{11i}^2}{u_{11}^2}
-\frac{2}{\alpha^2}( x_i)^2.
\end{eqnarray*}
Thus,
\begin{eqnarray*}
\sigma_{2}^{ij}P_{ij}&\geq&\frac{2\alpha }{u}+(\frac{1}{14}-\frac{2}{\alpha})
\frac{\sigma^{11}_{2}u_{111}^2}{u_{11}^2}
+\frac{\sum_{i\neq 1}(2u_{11}
-(1+\frac{2}{\alpha})\sigma_{2}^{ii})u^{2}_{11i}}{u_{11}^2}
\\&&-\frac{2}{\alpha} \sigma_{2}^{ii}x_{i}^2
+(n-1)\sigma_1.
\end{eqnarray*}
In view of \eqref{co2-2}, if we choose $u_{11}$ bigger than some constant
$C(n, \alpha, A)$ (otherwise our lemma holds true automatically), we have
\begin{eqnarray*}
2u_{11}-(1+\frac{2}{\alpha})\sigma_{2}^{ii}&\geq& (1-\frac{2}{\alpha})u_{11}
-(1+\frac{2}{\alpha})(n-2)\bigg((n-1)^2 \sqrt{\frac{2A}{n(n-1)u_{11}}}+\frac{7(n-1)}{5u_{11}}\bigg)
\\&\geq& (1-\frac{2}{\alpha})u_{11}
-\frac{C(n, \alpha, A)}{\sqrt{u_{11}}}-C(n, \alpha, A)>0.
\end{eqnarray*}
Next, if we choose $\alpha$ large, we obtain at $x_0$
\begin{eqnarray*}
0\geq\sigma_{2}^{ij}P_{ij}&\geq&\frac{2\alpha }{u}
-\frac{2}{\alpha} \sigma_{2}^{ii}x_{i}^2
+(n-1)\sigma_1\\&\geq& \frac{2\alpha }{u}
+\frac{(n-1)}{2}\sigma_1\\&\geq& \frac{2\alpha }{u}
+\frac{(n-1)}{2}u_{11}.
\end{eqnarray*}
So, we obtain our Lemma.
\end{proof}

We now begin to prove Theorem \ref{main}.
\begin{proof}
The proof is standard \cite{LRW} \cite{Tr}.
Let $u$ be an entire solution of the equation \eqref{SQ}. For any constant $R>1$, we consider the set
$$\Omega_{R}=\{y \in \mathbb{R}^n: u(Ry)<R^2\}.$$
Let
$$u_{R}(y)=\frac{u(Ry)-R^2}{R^2}.$$
We consider the following Dirichlet problem:
\begin{equation}\label{SQ-1}
\left\{
\begin{aligned}
&\sigma_2(D^2v(x))=1, \quad x \in \Omega_R; \\
&v=0, \quad x \in \partial\Omega_R.
\end{aligned}
\right.
\end{equation}
Since $$D^2_{y} u_R=D^2_{x} u,$$
clearly, $u_R$ is a 2-convex solution of \eqref{SQ-1}and satisfies $\sigma_{3}(D^2u_{R})\geq -A$.
Applying Lemma \ref{le3}, so we have the estimates:
\begin{eqnarray}\label{est-}
(-u_R)^{\alpha}|D^2 u_R|\leq C.
\end{eqnarray}
Now using the quadratic growth condition in Theorem \ref{main}, we have
\begin{eqnarray*}
b|R y|^2-c\leq u(R y)\leq R^2,
\end{eqnarray*}
which implies
\begin{eqnarray*}
|y|^2\leq \frac{1+c}{b}.
\end{eqnarray*}
Thus, $\Omega_R$ is bounded. Hence the constant $\alpha$ and $C$ becomes a absolutely constant. We now
consider the domain
$$\Omega^{\prime}_{R}=\{y \in \mathbb{R}^n: u(Ry)<\frac{R^2}{2}\}\subset \Omega_R.$$
In $\Omega^{\prime}_{R}$, we have
$$u_R(y)\leq-\frac{1}{2}.$$
Hence, \eqref{est-} implies that
$$|D^2 u_R|\leq 2^{\alpha} C.$$
Note that,
$$D^2_{y} u_R=D^2_{x} u.$$
Thus, using the previous two formulas, we have
$$|D^2 u|\leq C, \quad \mbox{in} \quad \widetilde{\Omega}_{R}=\{x: u(x)<\frac{R^2}{2}\},$$
where $C$ is a absolutely constant. The arbitrary of $R$ implies the
above inequality holds true in all over $\mathbb{R}^n$. Using Evans-Krylov theory, we have
$$|D^2 u|_{C^{\alpha}(B_R)}\leq C(n, \alpha)\frac{|D^2 u|_{C^{0}(B_{2R})}}{R^{\alpha}}\leq \frac{C(n, \alpha)}{R^{\alpha}}.$$
Hence, we obtain our theorem by letting $R\rightarrow +\infty$.
\end{proof}

\section{The proof of Theorem \ref{main1}}

We first recall the following lemma which is a special case (with $f=1$) of Lemma 3 in \cite{Qiu} and
here we give a quick and simple proof along the line of of Lemma 3 in \cite{Qiu} for this special case,
which results in a refined form:
\begin{lemma}\label{le4}
If $u \in C^4(\mathbb{R}^3)$ is a 2-convex solution of the equation \eqref{SQ} in $\mathbb{R}^3$, we have
\begin{eqnarray*}
\sigma_{2}^{ij}(\log \Delta u(x))_{ij}\geq \frac{1}{25}\sigma_{2}^{ij}(\log \Delta u(x))_{i}(\log \Delta u(x))_{j}.
\end{eqnarray*}
\end{lemma}

\begin{proof}
Assume $\{u_{ij}\}$ is diagonal and $\{u_{ij}\}=diag\{\lambda_1, \lambda_2, \lambda_3\}$.
Differential the equation \eqref{SQ} in $k$-th variable,
\begin{eqnarray}\label{3-1}
\sigma_{2}^{ij}u_{ijk}=0.
\end{eqnarray}
Taking twice derivative of the equation \eqref{SQ},
\begin{eqnarray*}
\sigma_{2}^{ij, pq}u_{ijk}u_{pqk}+\sigma_{2}^{ij}u_{ijkk}=0,
\end{eqnarray*}
which means
\begin{eqnarray}\label{3-2}
\sigma_{2}^{ij}u_{ijkk}=\sum_{i\neq j}u^{2}_{ijk}-\sum_{i\neq j}u_{iik}u_{jjk}.
\end{eqnarray}
Using \eqref{3-2}, we obtain
\begin{eqnarray*}
\Lambda:&=&\sigma_{2}^{ij}(\log \Delta u(x))_{ij}-\epsilon\sigma_{2}^{ij}(\log \Delta u(x))_{i}(\log \Delta u(x))_{j}
\\&=&\frac{\sigma_{2}^{ij}(\Delta u)_{ij}}{\sigma_1}
-\frac{(1+\epsilon)\sigma_{2}^{ii}(\sum_{k}u_{kki})^2}{\sigma_{1}^{2}}
\\&=&\frac{\sum_{i}(\sum_{k\neq p}u^{2}_{kpi}-\sum_{k\neq p}u_{kki}u_{ppi})}{\sigma_1}
-\frac{(1+\epsilon)\sigma_{2}^{ii}(\sum_{k}u_{kki})^2}{\sigma_{1}^{2}},
\end{eqnarray*}
where $\epsilon$ will be chosen later.
Noticing that
\begin{eqnarray*}
\sum_{i}\sum_{k\neq p}u^{2}_{kpi}
&=&2\sum_{k\neq p}u^{2}_{kpp}+\sum_{k\neq p\neq i}u^{2}_{kpi}
\\&=&2\sum_{k\neq p}u^{2}_{kpp}+6u^{2}_{123}
\\&\geq&2(u^{2}_{211}+u^{2}_{311}+u^{2}_{122}+u^{2}_{322}+u^{2}_{133}+u^{2}_{233}),
\end{eqnarray*}
\begin{eqnarray}\label{s-1}
\sum_{i}\sum_{k\neq p}u_{kki}u_{ppi}
&=&2\sum_{k\neq p}u_{kkp}u_{ppp}+\sum_{k\neq p\neq i}u_{kki}u_{ppi}
\nonumber\\&=&2\sum_{k\neq p}u_{kkp}u_{ppp}+2u_{113}u_{223}
+2u_{112}u_{332}+2u_{221}u_{331},
\end{eqnarray}
and
\begin{eqnarray}\label{s-2}
\sigma_{2}^{ii}\sum_{k}(u_{kki})^2
=\frac{(\sum_{k\neq i}\sigma_{2}^{ii}u_{kki}-\sigma_{2}^{ii}u_{iii})^2}{\sigma_{2}^{ii}}.
\end{eqnarray}
Then, using \eqref{3-1} to substitute terms with $u_{iii}$ in \eqref{s-1} and \eqref{s-2}, we have
\begin{eqnarray*}
\Lambda&\geq&\frac{2(u^{2}_{211}+u^{2}_{311}+u^{2}_{122}+u^{2}_{322}+u^{2}_{133}+u^{2}_{233})}{\sigma_1}
\\&&-\frac{2(u_{221}+u_{331})(-\sigma^{22}_{2}u_{221}-\sigma^{33}_{2}u_{331})}{\sigma_1\sigma_{2}^{11}}
\\&&-\frac{2(u_{112}+u_{332})(-\sigma^{11}_{2}u_{112}-\sigma^{33}_{2}u_{332})}{\sigma_1\sigma_{2}^{22}}
\\&&-\frac{2(u_{113}+u_{223})(-\sigma^{11}_{2}u_{113}-\sigma^{22}_{2}u_{223})}{\sigma_1\sigma_{2}^{33}}
\\&&-\frac{2(u_{113}u_{223}
+u_{112}u_{332}+u_{221}u_{331})}{\sigma_{1}}
\\&&-\sum_{i}\frac{(1+\epsilon)[\sum_{k\neq i}(\sigma_{2}^{ii}-\sigma_{2}^{kk})u_{kki}]^2}
{\sigma_{1}^{2}\sigma_{2}^{ii}}.
\end{eqnarray*}
Due to symmetry, we only to give the lower bound of the terms which
contain $u_{221}$ and $u_{331}$. We denote these terms by $\Lambda_1$
\begin{eqnarray*}
\Lambda_1&=&\frac{2(u^{2}_{311}+u^{2}_{221})}{\sigma_1}+\frac{2\sigma^{22}_{2}u_{221}^2}{\sigma_1\sigma_{2}^{11}}
+\frac{2\sigma^{33}_{2}u_{331}^2}{\sigma_1\sigma_{2}^{11}}
\\&&+\frac{2(\sigma^{22}_{2}+\sigma^{33}_{2})u_{221}u_{331}}{\sigma_1\sigma_{2}^{11}}-\frac{2u_{221}u_{331}}{\sigma_{1}}
\\&&-\frac{(1+\epsilon)[(\lambda_2-\lambda_1)u_{221}+(\lambda_3-\lambda_1)u_{331}]^2}
{\sigma_{1}^{2}\sigma_{2}^{11}}.
\end{eqnarray*}
Then, we can rewrite $\Lambda_1$ as a quadratic polynomial of $u_{221}$ and $u_{331}$:
\begin{eqnarray*}
\sigma_{1}^{2}\sigma_{2}^{11}\Lambda_1&=&
[2\sigma_1(\sigma_1+\lambda_3)-(1+\epsilon)(\lambda_1-\lambda_2)^2]u_{221}^2
\\&&+[2\sigma_1(\sigma_1+\lambda_2)-(1+\epsilon)(\lambda_1-\lambda_3)^2]u_{331}^2
\\&&+2[2\sigma_1 \lambda_1-(1+\epsilon)(\lambda_2-\lambda_1)(\lambda_3-\lambda_1)]u_{221}u_{331}.
\end{eqnarray*}
To show $\Lambda_1\geq 0$, we need to check
\begin{eqnarray*}
&&[2\sigma_1(\sigma_1+\lambda_3)-(1+\epsilon)(\lambda_1-\lambda_2)^2]\\&\times&
[2\sigma_1(\sigma_1+\lambda_2)-(1+\epsilon)(\lambda_1-\lambda_3)^2]\\&\geq&[2\sigma_1 \lambda_1-(1+\epsilon)(\lambda_2-\lambda_1)(\lambda_3-\lambda_1)]^2.
\end{eqnarray*}
Clearly, using $\lambda_1\lambda_2+\lambda_2\lambda_3+\lambda_3\lambda_1=1$, we have
\begin{eqnarray*}
&&2\sigma_1(\sigma_1+\lambda_3)-(1+\epsilon)(\lambda_1-\lambda_2)^2\\&=&
(1-\epsilon)\lambda_{1}^{2}+(1-\epsilon)\lambda_{2}^{2}+4\lambda_{3}^{2}+2\epsilon\lambda_{1}\lambda_{2}+6,
\end{eqnarray*}
\begin{eqnarray*}
&&2\sigma_1(\sigma_1+\lambda_2)-(1+\epsilon)(\lambda_1-\lambda_3)^2\\&=&
(1-\epsilon)\lambda_{1}^{2}+(1-\epsilon)\lambda_{3}^{2}+4\lambda_{2}^{2}+2\epsilon\lambda_{1}\lambda_{3}+6
\end{eqnarray*}
and
\begin{eqnarray*}
&&2\sigma_1 \lambda_1-(1+\epsilon)(\lambda_2-\lambda_1)(\lambda_3-\lambda_1)
\\&=&(1-\epsilon)\lambda_{1}^{2}-2(2+\epsilon)\lambda_{2}\lambda_{3}+3+\epsilon.
\end{eqnarray*}
So we only need to check
\begin{eqnarray*}
&&[(1-\epsilon)\lambda_{1}^{2}+3+\epsilon]
[(5-\epsilon)\lambda_{2}^{2}+(5-\epsilon)\lambda_{3}^{2}
-2\epsilon\lambda_{2}\lambda_{3}+6]\nonumber\\&&
+[(1-\epsilon)\lambda_{2}^{2}+4\lambda_{3}^{2}+
2\epsilon\lambda_{1}\lambda_{2}+3-\epsilon][(1-\epsilon)\lambda_{3}^{2}+4\lambda_{2}^{2}+
2\epsilon\lambda_{1}\lambda_{3}+3-\epsilon]
\nonumber\\&&\geq-4(2+\epsilon)[(1-\epsilon)\lambda_{1}^{2}+3+\epsilon]\lambda_2\lambda_3
+4(2+\epsilon)^2\lambda_{2}^{2}\lambda_{3}^{2},
\end{eqnarray*}
which is equivalent to
\begin{eqnarray}\label{la}
&&[(1-\epsilon)\lambda_{1}^{2}+3+\epsilon]
[(5-\epsilon)\lambda_{2}^{2}+(5-\epsilon)\lambda_{3}^{2}
+2(4+\epsilon)\lambda_{2}\lambda_{3}+6]\nonumber\\&&
+[(1-\epsilon)\lambda_{2}^{2}+4\lambda_{3}^{2}+
3-\epsilon][(1-\epsilon)\lambda_{3}^{2}+4\lambda_{2}^{2}+3-\epsilon]
\nonumber\\&&+
2\epsilon\lambda_{1}\lambda_{2}[(1-\epsilon)\lambda_{3}^{2}+4\lambda_{2}^{2}+
3-\epsilon]\nonumber\\&&
+2\epsilon\lambda_{1}\lambda_{3}[(1-\epsilon)\lambda_{2}^{2}+4\lambda_{3}^{2}+
3-\epsilon]\nonumber\\&&\geq4(2+\epsilon)^2\lambda_{2}^{2}\lambda_{3}^{2}
-4\epsilon^2\lambda_{1}^{2}\lambda_{2}\lambda_{3}.
\end{eqnarray}
Firstly, if we choose $0\leq\epsilon\leq\frac{1}{3}$, we get
\begin{eqnarray}\label{la-1}
&&[(1-\epsilon)\lambda_{1}^{2}+3+\epsilon]
[(5-\epsilon)\lambda_{2}^{2}+(5-\epsilon)\lambda_{3}^{2}
+2(4+\epsilon)\lambda_{2}\lambda_{3}+6]\nonumber\\&&\geq[(1-\epsilon)\lambda_{1}^{2}+3+\epsilon]
[(1-2\epsilon)\lambda_{2}^{2}+(1-2\epsilon)\lambda_{3}^{2}
+6]\nonumber\\&&\geq(1-\epsilon)(1-2\epsilon)\lambda_{1}^{2}
[\lambda_{2}^{2}+\lambda_{3}^{2}
]\nonumber\\&&\geq-4\epsilon^2\lambda_{1}^{2}\lambda_{2}\lambda_{3}.
\end{eqnarray}
Secondly, using $\lambda_{1}\lambda_{2}+\lambda_{1}\lambda_{3}=-\lambda_{2}\lambda_{3}$
to substitute $\lambda_{1}\lambda_{2}+\lambda_{1}\lambda_{3}$ with $\lambda_{2}\lambda_{3}$, we have
\begin{eqnarray}\label{la-2}
&&2\epsilon\lambda_{1}\lambda_{2}[(1-\epsilon)\lambda_{3}^{2}+4\lambda_{2}^{2}+
3-\epsilon]\nonumber\\&&
+2\epsilon\lambda_{1}\lambda_{3}[(1-\epsilon)\lambda_{2}^{2}+4\lambda_{3}^{2}+
3-\epsilon]\nonumber\\&=&2\epsilon(3-\epsilon)(1-\lambda_{2}\lambda_{3})
+2\epsilon(1-\epsilon)(1-\lambda_{2}\lambda_{3})\lambda_{2}\lambda_{3}\nonumber\\&&+
8\epsilon(1-\lambda_{2}\lambda_{3})(\lambda_{2}^{2}-\lambda_{2}\lambda_{3}+\lambda_{3}^{2}).
\nonumber\\&=&2\epsilon(3-\epsilon)+4\epsilon[2\lambda_{2}^2+2\lambda_{3}^2-3\lambda_{2}\lambda_{3}]
\nonumber\\&&-2\epsilon\lambda_{2}\lambda_{3}[4\lambda_{2}^2+4\lambda_{3}^2-(3+\epsilon)\lambda_{2}\lambda_{3}]
\nonumber\\&\geq&4\epsilon\lambda_{2}\lambda_{3}
+2\epsilon\lambda_{2}\lambda_{3}[4\lambda_{2}^2+4\lambda_{3}^2-(3+\epsilon)\lambda_{2}\lambda_{3}]
\nonumber\\&=&4\epsilon\lambda_{2}\lambda_{3}-2\epsilon(3+\epsilon)\lambda_{2}^2\lambda_{3}^2
+8\epsilon\lambda_{2}\lambda_{3}(\lambda_{2}^2+\lambda_{3}^2).
\end{eqnarray}
Lastly, if we choose $0\leq\epsilon\leq \sqrt{101}-10\approx 0.049$, we have
\begin{eqnarray*}
&&[(1-\epsilon)\lambda_{2}^{2}+4\lambda_{3}^{2}+
3-\epsilon][(1-\epsilon)\lambda_{3}^{2}+4\lambda_{2}^{2}+3-\epsilon]
\\&\geq&(3-\epsilon)(5-\epsilon)(\lambda_{2}^2+\lambda_{3}^2)
+[16+(1-\epsilon)^2]\lambda_{2}^{2}\lambda_{3}^{2}+4(1-\epsilon)(\lambda_{2}^4+\lambda_{3}^4)
\\&\geq&(3-\epsilon)(5-\epsilon)(\lambda_{2}^2+\lambda_{3}^2)
+[16+(1-\epsilon)^2-8\epsilon]\lambda_{2}^{2}\lambda_{3}^{2}+4\epsilon(\lambda_{2}^2+\lambda_{3}^2)^2.
\\&\geq&4\epsilon\lambda_{2}\lambda_{3}-2\epsilon(3+\epsilon)\lambda_{2}^2\lambda_{3}^2
+4(2+\epsilon)^2\lambda_{2}^{2}\lambda_{3}^{2}+8\epsilon\lambda_{2}\lambda_{3}(\lambda_{2}^2+\lambda_{3}^2),
\end{eqnarray*}
which together with \eqref{la-1} and \eqref{la-1} implies \eqref{la} holds true, if we choose $\epsilon=\frac{1}{25}$.
\end{proof}

Then, we have the following estimate.
\begin{lemma}\label{le5}
Let $\Omega$ be a bounded open set in $\mathbb{R}^3$, we consider the Drichlet problem of the 2-Hessian equation
\begin{equation}\label{SQ-}
\left\{
\begin{aligned}
&\sigma_2(D^2u(x))=1, \quad x \in \Omega; \\
&u=0, \quad x \in \partial\Omega.
\end{aligned}
\right.
\end{equation}
Then, for any 2-convex solution $u \in C^4(\Omega)\cap C(\overline{\Omega})$, we have the following estimate
of Pogorelov type ,
\begin{eqnarray}\label{est}
\max_{x \in \Omega}(-u(x))^{\alpha}\max\{|D^2 u(x)|, 1\}\leq C
\end{eqnarray}
for sufficiently large $\alpha>0$. Here $\alpha$ and
 $C$ only depend on $n$, the diameter of the domain $\Omega$.
\end{lemma}

\begin{proof}
First, we translate the coordinate system such that $\Omega$ contains the coordinate origin.
Similar to the argument in Lemma \ref{le3}, using the Comparison principle
(see Theorem 17.1 in Page 443 of \cite{GT}), there exists the function
$$w=\frac{1}{2\sqrt{3}}|x|^2-a$$
such that
$$w\leq u\leq 0,$$
where $a$ depends on the diameter of the domain.
Thus, $|u|$ can be controlled by the diameter of the domain $\Omega$.

Since $u$ is 2-convex, then $\sigma_{2}^{ij}$ is positive definite. So the Hessian estimates can
be reduced to the estimate of $\Delta u$ due to the following fact
$$\max_{\xi \in \mathbb{S}^{n-1}}|u_{\xi\xi}|\leq \Delta u.$$
We consider the auxiliary function in $\Omega$
\begin{eqnarray}\label{P}
P(x, \xi)=\alpha\log (-u)+\log \max\{\Delta u, 1\}+\frac{1}{2}|x|^2.
\end{eqnarray}
Since $u=0$ on $\partial \Omega$, so the maximum of $P$ is attained
in some interior point $x_0$ of $\Omega$. Assume $\Delta u\geq 2$ at $x_0$
(otherwise our lemma holds true automatically). Thus at $x_0$
\begin{eqnarray}\label{1-diff-}
0=P_i=\frac{\alpha u_i}{u}+(\log \Delta u)_i+ x_i.
\end{eqnarray}
Contracting
\begin{eqnarray*}
P_{ij}=\frac{\alpha u_{ij}}{u}-\frac{\alpha u_i u_j}{u^2}
+(\log \Delta u)_{ij}+\delta_{ij}
\end{eqnarray*}
with $\sigma_{2}^{ij}$ gives at $x_0$
\begin{eqnarray*}
0\geq \sigma_{2}^{ij}P_{ij}&=&\frac{2\alpha}{u}
-\frac{\alpha\sigma_{2}^{ii}u_{i}^2}{u^2}
+\sigma_{2}^{ij}(\log \Delta u)_{ij}+2\sigma_1
\end{eqnarray*}
in view of $\sigma_2=1$. Using \eqref{1-diff-}, we have at $x_0$
\begin{eqnarray*}
-(\frac{u_i}{u})^2\geq -\frac{2}{\alpha^2}((\log \Delta u)_i)^2
-\frac{2}{\alpha^2}( x_i)^2,
\end{eqnarray*}
which implies together with Lemma \ref{le4}
\begin{eqnarray*}
\sigma_{2}^{ij}P_{ij}&\geq&\frac{2\alpha }{u}+(\frac{1}{25}
-\frac{2}{\alpha})\sigma_{2}^{ii}((\log \Delta u)_i)^2
-\frac{2}{\alpha} \sigma_{2}^{ii}x_{i}^2
+2\sigma_1.
\end{eqnarray*}
If we choose $\alpha\geq 50$,
\begin{eqnarray*}
\frac{1}{25}-\frac{2}{\alpha}\geq0,
\end{eqnarray*}
which implies at $x_0$
\begin{eqnarray*}
\sigma_{2}^{ij}P_{ij}&\geq&\frac{2\alpha }{u}
-\frac{2}{\alpha} \sigma_{2}^{ii}x_{i}^2
+2\sigma_1.
\end{eqnarray*}
Then, if we choose $\alpha$ again large enough, we have we obtain at $x_0$
\begin{eqnarray*}
0\geq\sigma_{2}^{ij}P_{ij}\geq \frac{2\alpha }{u}
+\Delta u.
\end{eqnarray*}
So, we obtain our Lemma.
\end{proof}

Using the above Lemma, following almost the same proof as Theorem \ref{main}, we can obtain Theorem
\ref{main1}.

\end{document}